%% file: Lambert-ImPJoyce.tex
\documentclass[12pt,a4paper]{scrartcl}

\usepackage{ifpdf}
\usepackage{color}
\usepackage{latexsym}                   
\usepackage{epsfig}                     
\usepackage{graphicx}
\usepackage{amssymb, amsmath, amsthm,bbm}   
\usepackage{enumerate}                  
\usepackage{mathrsfs}                   
\usepackage{geometry}                                                                          
\usepackage{verbatim}                                                                           
\usepackage{tabularx}     

\usepackage{setspace}

\newtheorem{thm}{Theorem}[]
\newtheorem{lem}[thm]{Lemma}

\newtheorem{cor}[thm]{Corollary}
\newtheorem{prop}[thm]{Proposition}

\theoremstyle{definition}

\theoremstyle{definition}

\theoremstyle{definition}

     {\begin{list}{}%
             {\setlength{\leftmargin}{#1}}%
             \item[]%
     }
     {\end{list}}

\newcommand{\citep}{\cite}
\newcommand{\citet}{\cite}

\begin{document}

\title{The genealogy of a sample from a binary branching process
}
\subtitle{In memoriam Paul Joyce}
\author{Amaury Lambert$^{1,2}$
}
\date{\today}
\maketitle
\noindent\textsc{$^1$
Laboratoire de Probabilit\'es et Mod\`eles Al\'eatoires (LPMA), UPMC Univ Paris 06, CNRS\\
Paris, France}\\
\noindent\textsc{$^2$
Center for Interdisciplinary Research in Biology (CIRB), Coll\`ege de France, CNRS, INSERM, PSL Research University\\
Paris, France}\\
\textsc{E-mail: }amaury.lambert@upmc.fr\\
\textsc{URL: }http://www.lpma-paris.fr/pageperso/amaury.lambert/\\

\doublespacing

\section*{\textsc{Abstract}}
At time 0, start a time-continuous binary branching process, where particles give birth to a single particle independently (at a possibly time-dependent rate) and die independently (at a possibly time-dependent and age-dependent rate). A particular case is the classical birth--death process. Stop this process at time $T>0$. It is known that the tree spanned by the $N$ tips alive at time $T$ of the tree  thus obtained (called reduced tree or coalescent tree) is a coalescent point process (CPP), which basically means that the depths of interior nodes are iid. Now select each of the $N$ tips independently with probability $y$ (Bernoulli sample). It is known that the tree generated by the selected tips, which we will call Bernoulli sampled CPP, is again a CPP. Now instead, select exactly $k$ tips uniformly at random among the $N$ tips ($k$-sample). We show that the tree generated by the selected tips is a mixture of Bernoulli sampled CPPs with the same parent CPP, over some explicit distribution of the sampling probability $y$. An immediate consequence is that the genealogy of a $k$-sample can be obtained by the realization of $k$ random variables, first the random sampling probability $Y$ and then the $k-1$ node depths which are iid conditional on $Y=y$.

\bigskip
\noindent
\textit{Running head.} The genealogy of a sample from a binary branching process.\\
\textit{Key words and phrases.}  Splitting tree; random tree; birth--death process; incomplete sampling; subsampling; coalescent point process; finite exchangeable sequence.


\section*{\textsc{Introduction}}

\subsection*{Model and objective of the paper}

In this work, we consider a binary branching process in continuous time, possibly non-Markovian, that has the following properties, further denoted $(\star)$. 
\begin{itemize}
\item
The process starts with one particle at time 0;
\item At any time $t$, particles give birth independently at rate $\lambda(t)$, to a single daughter particle at each birth event;
\item At any time $t$, particles with age $x$ independently die at rate $\mu(t,x)$;
\item The process is stopped at time $T>0$ and is conditioned to have $N\ge 1$ particles alive at time $T$.   
\end{itemize}
This process generates a discrete metric tree, called \emph{splitting tree} \cite{GK97, L10}, with origin at time 0 and  $N$ tips at distance $T$ from the root, that we  call \emph{extant tips}. The inherent asymmetry between mother and daughter endows the splitting tree with a natural \emph{plane orientation}, where daughters sprout to the right of their mother, see Figure \ref{fig:examplePhylo}a. 

An oriented, ultrametric tree with $N$ tips is characterized by its \emph{node depths}, or \emph{coalescence times}, $H_1, \ldots, H_{N-1}$, as in Figure \ref{fig:examplePhylo}b. The orientation of the tree implies that $H_i$ ($1\le i\le n-1$) is the coalescence time between extant tip $i-1$ and extant tip $i$, where tips are labelled $0,\ldots, n-1$ from  left to right in the plane orientation, and also that  $\max \{H_{i+1},\ldots, H_j\}$ is the coalescence time between tip $i$ and tip $j$.

\begin{figure}[!ht]
\input{examplePhylo}
\caption{a) A plane oriented tree generated by a branching process; the $N=9$ particles extant at $T$ are labelled $0,1, \ldots , 8$ from left to right; b) 
The reduced tree obtained from the full tree in a), showing the coalescence times $H_4$, between extant tips 3 and 4, and $H_5$ between extant tips 4 and 5.}
\label{fig:examplePhylo}
\end{figure}
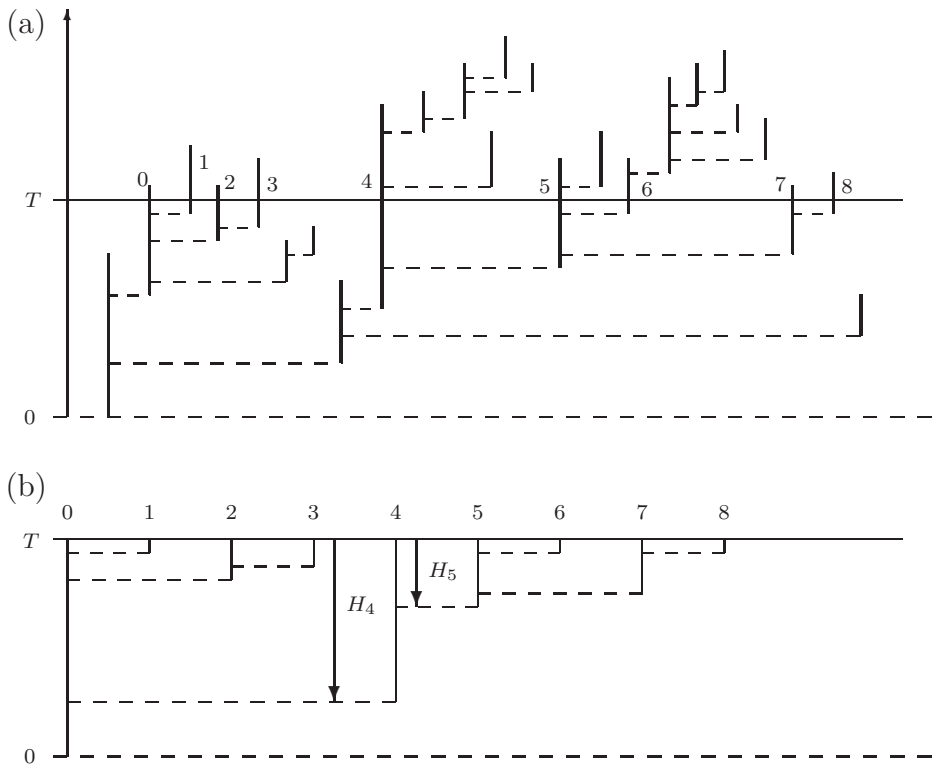

When $\mu(t,x)$ does not depend on age $x$, the branching process is merely a birth--death process with \textit{per capita} birth rate $\lambda(t)$ and death rate $\mu(t)$. In this case, it is actually equivalent to select uniformly at each birth event which lineage is the mother and which lineage is the daughter.

We are interested in the so-called \emph{reduced tree}, also called \emph{coalescent tree} in population genetics and \emph{reconstructed tree} in phylogenetics, i.e., the tree generated by the extant tips of the splitting tree, which is the genealogy of the particles alive at time $T$. This tree is said \emph{ultrametric} with height $T$, in the sense that all its tips are at the same distance $T$ to the root (i.e., all its tips are extant tips), so that the metric induced by the tree metric on its tip set is ultrametric (see e.g. \cite{L17a, L17b}).

More specifically, we are interested in the tree generated by a sample of the extant tips of the splitting tree, or equivalently, in the tree generated by a sample from (the tips of) the reduced tree. In the biology literature there are mainly two classical sampling schemes (but other sampling schemes can be useful, like diversified sampling or higher-level sampling, see \cite{LS13}). The first scheme, called \emph{Bernoulli sampling} scheme, consists in selecting each extant tip independently with the same probability, say $y$.  The second scheme, called \emph{$k$-sampling} scheme, consists in drawing uniformly $k$ tips among the extant tips of the splitting tree conditioned upon $N\ge k$. The goal of the paper is to gather some known results about Bernoulli samples and to present new results for the genealogy of a $k$-sample, including an explicit de Finetti representation of node depths, that is, as a mixture of sequences of independent and identically distributed (iid) random variables \cite{A85, F31, HS55}.

\subsection*{Coalescent point processes}

A coalescent point process (CPP) with height $T$ is a random, oriented ultrametric tree with height $T$, whose node depths $H_1, \ldots, H_{N-1}$ form a sequence of independent copies of some r.v. $H>0$, stopped at its first value larger than $T$. 
Throughout the paper, we will assume that $H$ has a density denoted $f$ and we will use the notation
$$
F(t):=\frac 1{P(H>t)},
$$
so that 
$$
f=\frac{F'}{F^2} ,
$$
and we will say that the CPP has \emph{inverse tail distribution $F$}. 
The mundane consequence is that the number (again denoted) $N$ of extant tips in a CPP is always a shifted geometric r.v., namely $P(N=n)=(1-a) a^{n-1}$, where 
$$
a:=P(H<T) . 
$$ 
Now the likelihood of an ultrametric tree $\tau$ with $n$ tips and node depths $x_1<\cdots<x_{n-1}$ under the CPP distribution is simply 
\begin{equation}
\label{eqn:likelihood}
{\mathcal L}(\tau )=\frac{C(\tau)}{F(T)}\prod_{i=1}^{n-1} f(x_i) ,
\end{equation}
where $C(\tau)$ is a constant that depends whether $\tau$ is oriented or not \cite{L17a, LS13}. Specifically, $C(\tau) = 1$ if $\tau$ is oriented and $C(\tau) = 2^{n-1-\alpha(\tau)}$ if $\tau$ is non-oriented, where $\alpha(\tau)$ is the number of cherries of $\tau$ (a cherry is a pair of tips which are the only tips descending from their most recent common ancestor in $\tau$). Notice that the likelihood of $\tau$ conditional on $N=n$ is obtained  by dividing ${\mathcal L}(\tau)$ by $P(H<T)^{n-1}$.

In \cite{L10, LS13}, it was shown that the reduced tree of a splitting tree is a CPP with inverse tail distribution $F$ that can be characterized from the knowledge of the rates $\lambda$ and $\mu$, as in the following statement which merges Theorem 3 and Proposition 4 from \cite{LS13} (see also \cite{H12, H13, NMH94, P04} for similar, but partial results).
We first need to define for any $s\ge t$, the density $g(t,s)$ at time $s$ of the death time of a particle born at time $t$. Elementary properties of Poisson processes entail the following formula
\begin{equation}
\label{eqn:g}
g(t,s) = \mu(s, s-t) \exp\left\{-\int_t^s \mu(u, u-t)\, du\right\} .
\end{equation}
\begin{thm}
\label{prop:NEW}
The reduced tree at height $T$ of a splitting tree satisfying the properties $(\star)$ stated in the introduction is a CPP whose inverse tail distribution $F$  is the unique solution to the following integro-differential equation
\begin{equation}
\label{eqn:characF}
F'(t) = \lambda(T-t)\,\left( F(t) - \int_0^t \ F(s)\,g(T-t,T-s) \,ds \right)\qquad t\ge 0,
\end{equation}
with initial condition $F(0)=1$, where $g$ is given by \eqref{eqn:g}.
\end{thm}
In the case when $\mu(t,x)$ does not depend on $x$, $F$ is given by an explicit formula
$$
F(t) =1+\int_{T-t}^T \lambda(s)\,\exp\left\{\int_s^T  \,r(u)du\right\}ds,
$$
where $r(t) := \lambda(t) - \mu(t)$. When in addition rates do not depend on time, the branching process merely becomes a birth--death process with \textit{per capita} birth rate $\lambda$ and death rate $\mu$ and we get, writing $r=\lambda-\mu$,
\begin{equation}
\label{eqn:Markovian scale}
F(t)=\begin{cases}
1 + \frac{\lambda}{r}\big(e^{rt}- 1\big) & \text{if } r\not=0 \\
1+\lambda t & \text{if }r=0.
\end{cases}
\end{equation}
In particular, the common density of node depths is
\begin{equation}
f(t)=\frac{F'(t)}{F(t)^2}= \frac{ \lambda r^2  e^{-rt}}{( \lambda -\mu e^{-rt })^2   }\qquad t\ge 0.
\end{equation}
In the remainder of the paper, we focus on the genealogy of a sample taken from a CPP, which amounts to taking a sample from the extant tips of the splitting tree, provided $F$ is chosen as in \eqref{eqn:characF}, which boils down to \eqref{eqn:Markovian scale} in the case of a birth-death process.

\section*{The case of a Bernoulli sample}

In \cite{LS13}, it was further shown that a Bernoulli sampled CPP with sampling probability $y$ is again a CPP, with inverse tail distribution $F_y$ obtained from $F$ by the simple formula
\begin{equation}
\label{eqn:characFy}
F_y= 1-y + yF,
\end{equation}
and with density $f_y$ thus given by
$$
f_y= \frac{F_y'}{F_y^2} .
$$
In the special case of a birth--death process with \textit{per capita} birth rate $\lambda$ and death rate $\mu$, with $r=\lambda -\mu$, we can compute
\begin{equation}
\label{eqn:characfy}
f_y (t) =
 \frac{y \lambda r^2  e^{-rt}}{(y \lambda +(r-y\lambda)e^{-rt })^2   } \qquad t\ge 0,
\end{equation}
when $r\not=0$, and if $r=0$, $f_y(t) = y\lambda (1+y\lambda t)^{-2}$.

To see why \eqref{eqn:characFy} holds, first observe that the number of unsampled tips between two consecutively sampled tips is a geometric r.v. $G$ with success probability $y$. Also remember that the orientation of the CPP implies that the coalescence time between extant tip $i$ and extant tip $j$ is $\max \{H_{i+1},\ldots, H_j\}$, where $H_i$ is the coalescence time between tip $i-1$ and tip $i$. As a consequence, the genealogy of the sample of a CPP is again a CPP, where $H$ is replaced by the random variable $H_y$ distributed as the maximum of $G$ independent copies of $H$. Elementary calculations then entail that the inverse tail distribution $F_y$ of $H_y$ is given by \eqref{eqn:characFy}. Also note that we recover $F$ when taking $y=1$, which amounts to full sampling.

The bottomline is that the tree generated by a Bernoulli sample of a splitting tree, which is also the tree generated by a Bernoulli sample of its reduced tree, is a CPP with inverse tail distribution $F_y$, where $F$ is given by  \eqref{eqn:characF}.

Actually, it is shown more generally in \cite{LS13} that the tree remains a CPP even in the presence of bottlenecks occurring at fixed times $s_1<\cdots <s_k= T$, that is when the whole descendance of each lineage extant at time $s_i$ is independently killed with the same probability $1-y_i$, $i=1,\ldots,k$ (here, the $k$-th bottleneck corresponds to Bernoulli sampled tips with probability $y_k$). We will not carry on this aspect.

To simulate the genealogy of a Bernoulli sample of a splitting tree, it is thus enough to simulate a geometric number of iid random variables distributed as $H_y$, rather than na\"{i}vely run the whole branching process forward-in-time up until time $T$, select a sample at time $T$  and prune all extinct lineages. Thus, the CPP representation of subsampled reduced trees allows for fast simulation algorithms of these trees, much faster than the na\"{i}ve simulation, in particular in the presence of high death rates.

As for parameter inference, observe that the likelihood of an ultrametric  tree $\tau$ with $k$ tips and node depths $x_1<\cdots<x_{k-1}$ under the Bernoulli sampling scheme is particularly simple
\begin{equation}
\label{eqn:likelihood-Bernoulli}
{\mathcal L}(\tau \mid y )=\frac{C(\tau)}{F_y(T)}\prod_{i=1}^{k-1} f_y(x_i) ,
\end{equation}
where $C(\tau)$ has been defined earlier. Notice that the likelihood of $\tau$ conditional on $N=n$ is obtained  by dividing ${\mathcal L}(\tau \mid y )$ by $P(H_y<T)^{n-1}$. As a consequence, as soon as the common density $f_y$ of node depths is explicit or can be quickly computed numerically, the parameters of the model (sampling probability $y$, height $T$, birth and death rates or a parameterization thereof) can be estimated pointwise by maximum likelihood or their posterior distribution can be computed in a Bayesian framework.

\section*{The case of a $k$-sample}

In contrast to the case of a Bernoulli sample, our efforts to express the likelihood of the genealogy of a $k$-sample have resulted so far in rather opaque formulas that we briefly present in the next section. 

\subsection*{Law of the tree conditional on $N$}

Let us define  ${\mathcal L}(\tau, m \mid k)$ as the likelihood of an ultrametric tree $\tau$ with height $T$ under the $k$-sampling scheme on the event that the full tree has total number of tips $N$ equal to $k+m$. If the tree $\tau$ of the $k$-sample has node depths $x_1<\cdots< x_{k-1}$, writing $x_k=T$, we have
\begin{equation}
\label{eqn:likelihoodmissing}
{\mathcal L}(\tau,m\mid k)={\mathcal L}(\tau )\ {m+k\choose k}^{-1}\ \sum_{\vec{m}:m_1+\cdots+m_k=m}  \prod_{i=1}^{k}(m_i+1)P(H<x_i)^{m_i},
\end{equation}
where ${\mathcal L}(\tau )=\frac{C(\tau)}{F(T)}\prod_{i=1}^{k-1} f(x_i)$ is given by \eqref{eqn:likelihood} and the sum is taken over all possible configurations of samples of cardinal $k$ among $k+m$ extant tips, that is, $m_i\ge 0$ ($1\le i\le k-1$) is the number of unsampled tips between sampled tip $i-1$ and sampled tip $i$. The term $m_k$ gathers the unsampled tips before the first sampled tip and after the last one. The multiplicative terms $m_i+1$  are due to summing over all possible positions of the largest node depth between sampled tip $i-1$ and sampled tip $i$. 

For practical purposes, especially in the case of large $m$, the sum in \eqref{eqn:likelihoodmissing} is not computable numerically in short time. We will now use an analytical trick to express the $(k-1)$-dimensional distribution function of the node depths of the genealogy of the $k$-sample.
\begin{prop}
Let $H_1',\ldots, H_{k-1}'$ denote the node depths (in the order induced by the plane orientation) of the genealogy of a uniform $k$-sample from a CPP with typical node depth $H$ and height $T$. Then for any $m\ge 0$ and $x_1,\ldots, x_{k-1}\in (0,T)$, the $(k-1)$-dimensional distribution function of $H_1', \ldots, H_{k-1}'$ on the event that the full tree has $N=k+m$ tips, is given by 
\begin{multline}
\label{eqn:distribution-function}
P(N=k+m, H_1' <x_1, \ldots, H_{k-1}'<x_{k-1})\\
= \frac{1-p_0}{ {m+k\choose k}}\ \left(\prod_{i=1}^{k-1}p_i \right)\sum_{\ell= 0}^{m}(\ell+1)p_0^{\ell}\sum_{i=1}^{k-1}\frac{p_i^{m-\ell+k-2}}{\prod_{j=1,\ldots, k-1, j\not=i} (p_i-p_j)}\\
= \frac{1-p_0}{ {m+k\choose k}}\ \left(\prod_{i=1}^{k-1}p_i \right)\sum_{i=1}^{k-1}\frac{p_i^{k-2}}{\prod_{j=1,\ldots, k-1, j\not=i} (p_i-p_j)}  \frac{p_i^{m+2}-(m+2)p_ip_0^{m+1}+ (m+1)p_0^{m+2}}{(p_i-p_0)^2}
\end{multline}
where $p_0:= P(H<T)$ and $p_i :=P(H<x_i)$, $i=1,\ldots, k-1$.
\end{prop}

\begin{proof}
First observe (directly, or by performing $k-1$ integrations in \eqref{eqn:likelihoodmissing}), that the probability in the statement equals
$$
 P(H>T)\ {m+k\choose k}^{-1}\ \sum_{\vec{m}:m_1+\cdots+m_k=m} (m_k+1)P(H<T)^{m_k}  \prod_{i=1}^{k-1}P(H<x_i)^{m_i+1}.
$$
The trick is to use the fact (see Lemma 10 in \cite{LS13}) that 
$$
\sum_{(m_1, \ldots, m_n)\in N_m^n}\prod_{i=1}^n p_i^{m_i}= \sum_{i=1}^n \frac{p_i^{m+n-1}}{\prod_{j=1,\ldots, n, j\not=i} (p_i-p_j)},
$$
where $N_m^n$ is the set of $n$-tuples of non-negative integers $(m_1, \ldots, m_n)$  such that $\sum_{i=1}^n m_i = m$, and the $p_i$'s are pairwise distinct.
As a consequence, taking $n=k-1$ and $p_i= P(H<x_i)$, we get that $P(N=k+m, H_1' <x_1, \ldots, H_{k-1}'<x_{k-1})$ is equal to 
\begin{eqnarray*}
 &=& P(H>T)\ {m+k\choose k}^{-1}\ \sum_{m_k= 0}^{m}(m_k+1)P(H<T)^{m_k}\sum_{\vec{m}:m_1+\cdots+m_{k-1}=m-m_k} \prod_{i=1}^{k-1}p_i^{m_i+1}\\
 &=& P(H>T)\ {m+k\choose k}^{-1}\ \sum_{m_k= 0}^{m}(m_k+1)P(H<T)^{m_k}\left(\prod_{i=1}^{k-1}p_i \right)\sum_{i=1}^{k-1}\frac{p_i^{m-m_k+k-2}}{\prod_{j=1,\ldots, k-1, j\not=i} (p_i-p_j)}.
\end{eqnarray*}
Replacing the index $m_k$ by $\ell$ and $P(H<T)$ by $p_0$ yields the first formula in the statement of the proposition. For the second equality, we have computed the term $\sum_{\ell= 0}^{m}(\ell+1)(p_0/p_i)^{\ell} $ by using the fact that
$$
\sum_{\ell= 0}^{m}(\ell+1)z^{\ell} = \frac{d}{dz}\sum_{\ell= 0}^{m+1}z^{\ell} =  \frac{d}{dz}\frac{1-z^{m+2}}{1-z} 
= \frac{1-(m+2)z^{m+1}+ (m+1)z^{m+2}}{(1-z)^2},
$$
and the proof ends with elementary calculations.
\end{proof}
In the next section, we show a much more transparent formulation than the density \eqref{eqn:likelihoodmissing} and the distribution function \eqref{eqn:distribution-function} of node depths. We do not get to a simpler formula for the density ${\mathcal L}(\tau, m \mid k)$, but we do have one for the likelihood of the tree of the $k$-sample 
$$
{\mathcal L}(\tau \mid k) = \sum_{m\ge 0} {\mathcal L}(\tau, m \mid k),
$$
which can be expressed thanks to an elegant de Finetti representation. More specifically, we show that the tree generated by the $k$-sample is a mixture of Bernoulli sampled CPPs with the same parent CPP, over some explicit distribution of the sampling probability $y$.

\subsection*{Main result}
\begin{thm}
\label{thm:main}
Fix $k\ge 1$.
 Let $\Pi_k$ denote the law of the tree generated by $k$ tips sampled uniformly from the tips of a standard CPP (conditioned to have at least $k$ tips) and for each $y\in (0,1)$ let $\Gamma_{y,k}$ denote the law of the Bernoulli sampled CPP with sampling probability $y$, conditioned to have $k$ tips. Then
\begin{equation}
\Pi_k=\int_0^1\mu_k(dy) \ \Gamma_{y,k}
\end{equation}
where $\mu_k$ is the probability distribution defined by
\begin{equation}
\label{eqn:muk}
\mu_k(dy):= \frac{k(1-a)y^{k-1} }{\left(1-a(1-y)\right)^{k+1}}\, dy \qquad y\in(0,1),
\end{equation}
and $a=P(H<T)$. 
\end{thm}
\noindent
The next statement is an elementary consequence of Theorem \ref{thm:main}.
\begin{cor}
\label{cor:cor1}
The ultrametric tree with law $\Pi_k$ can be obtained by first drawing $Y$ from the probability distribution $\mu_k$ and then conditional on $Y=y$, by drawing the $k-1$ node depths of the Bernoulli sampled CPP with sampling probability $y$ conditioned upon $k$ tips, which are iid r.v.'s with density on $(0,T)$ given by $c^{-1}f_y$, where 
$f_y=F_y'/F_y^2$, $F_y=1-y+yF$ and $c$ is the normalizing constant given by 
$$
c:= 1-\frac{1}{F_y(T)}= \frac{ay }{1-a(1- y)} 
$$
\end{cor}
\noindent
The next corollary is a consequence of the previous one.
\begin{cor}
\label{cor:cor2}
The likelihood of a tree $\tau$ with law $\Pi_k$ and node depths $x_1<\cdots <x_{k-1}$ is given by
\begin{multline}
\label{eqn:likelihood-k}
{\mathcal L}(\tau \mid k)=C(\tau)\int_0^1\mu_k(dy)\prod_{i=1}^{k-1} c^{-1}f_y(x_i) \\=C(\tau)\frac{k(1-a)}{a^{k-1}} \int_0^1 \frac{dy}{\left(1-a(1-y)\right)^{2}}\,\prod_{i=1}^{k-1} f_y(x_i) .
\end{multline}
\end{cor}
\noindent
We now introduce some notation in order to prove Theorem \ref{thm:main}.

Recall that the total number $N$ of tips in the standard CPP is a shifted geometric r.v. with failure probability $a=P(H<T)$, that is $P(N=n) = (1-a)a^{n-1}$ for all $n\ge 1$. Now conditional on $N\ge k$, select uniformly at random a subset $S$ of $k$ elements of $\{1,\ldots, N\}$ and define $I_i:=\mathbbm{1}_{i\in S}$. 

Now fix $y\in (0,1)$ and let $M$ be a r.v. equally distributed as $N$. Next let $J_i(y)$ be iid Bernoulli r.v.'s with parameter $y$ independent of $M$ and let $K(y)$ be the number of labels $i\le M$ such that $J_i(y)=1$.
Let $P_{y}$ denote the joint law of the r.v.'s $(J_i(y))_i$, $M$ and $K(y)$. Define the positive measure $Q$ by
$$
Q:=\int_0^1 dy\, y^{-1}\,P_y 
$$
Theorem \ref{thm:main} is a consequence of the following lemma.
\begin{lem}
The joint law of $N$ and  $(I_i)_{1\le i \le N}$ under $P(\cdot\mid N\ge k)$ is identical to the joint law of $M$ and $(J_i)_{1\le i \le M}$ under $Q(\cdot\mid K=k)$.
We have
$$
Q(Y\in dy\mid K=k) = \frac{k(1-a)y^{k-1} }{\left(1-a(1-y)\right)^{k+1}} \, dy ,
$$
and conditional on $Y=y$ (that is, under $Q(\cdot \mid Y=y, K=k)$), the sequence $(J_i)$ stopped at its $k$-th success forms a sequence of iid Bernoulli r.v. with parameter $1-a(1-y)$ stopped at its $k$-th success.
\end{lem}
\begin{proof}
It is obvious that under $Q(\cdot\mid M=n, K=k)$, the labels $i$ such that $J_i=1$ are uniformly distributed in $\{1,\ldots, n\}$, so we only have to show that 
$$
Q(M=n\mid K=k) = P(N=n\mid N\ge k).
$$
We have
$$
Q(M=n, K=k) = \int_0^1 \frac{dy} y (1-a) a^{n-1} {n \choose k} y^k(1-y)^{n-k} = (1-a) a^{n-1} {n \choose k} \int_0^1  y^{k-1}(1-y)^{n-k} dy.
$$
Because the last integral equals 
$\displaystyle
\frac{(k-1)! \, (n-k)!}{n!}
$,
we get 
$$
Q(M=n, K=k) = (1-a) \frac{a^{n-1}}k.
$$
Summing over $n\ge k $ yields 
$$
Q(K=k) = \frac{a^{k-1}}k
$$
and the ratio of the last two displayed quantities is 
$$
\frac{Q(M=n, K=k)}{Q(K=k)} = Q(M=n\mid K=k) = a^{n-k} = P(N=n\mid N\ge k).
$$
Now let us show that under $Q(\cdot\mid Y=y, K=k)$, the sequence $(J_i)$ stopped at its $k$-th success is a sequence of Bernoulli r.v.'s with failure probability $a(1-y)$. Set $G_j$ for the difference between the $j$-th 1 and the $(j+1)$-st 1 in the sequence $(J_i)$. Rigorously, we set
$$
L_j:= \min\{i> L_{j-1}: J_i=1\}, \qquad 1\le j\le k,
$$
with $L_0:=0$, and
$$
G_j:= L_j - L_{j-1},\qquad 1\le j\le k+1,
$$
with $L_{k+1}=M$. Note that $L_j\ge 1$ for all $j\le k$, whereas $G_{k+1}$ can be zero (when $L_k=M$).
Then for all $y\in(0,1)$, $k\ge 1$, $g_j\ge 1$, $j=1,\ldots,k$ and $g_{k+1}\ge 0$, denoting $n=\sum_{j=1}^{k+1}g_j$, we have
$$
Q(Y\in dy, K=k, G_j = g_j, j=1,\ldots,k+1) = \frac{dy} y (1-a)a^{n-1}\left( \prod_{j=1}^{k} (1-y)^{g_j-1}y\right) (1-y)^{g_{k+1}} .
$$
This can also be written
$$
Q(Y\in dy, K=k, G_j = g_j, j=1,\ldots,k+1) = \frac{dy} y (1-a)a^{k-1}\left( \prod_{j=1}^{k} \big(a(1-y)\big)^{g_j-1}y\right) \big(a(1-y)\big)^{g_{k+1}} ,
$$
so that summing over the $g_j$'s, we get
$$
Q(Y\in dy, K=k) = (1-a)     \frac{(a y)^{k-1}}{\left(1-a(1-y)\right)^{k+1}} \, dy ,
$$
and dividing by $Q(K=k)$, we get
$$
Q(Y\in dy\mid K=k) = k (1-a)\frac{y^{k-1}}{\left(1-a(1-y)\right)^{k+1}} \, dy .
$$
In addition,
\begin{multline*}
Q(G_j = g_j, j=1,\ldots,k+1\mid Y=y, K=k) \\=  \left( \prod_{j=1}^{k} \big(a(1-y)\big)^{g_j-1}(1-a(1-y))\right) \big(a(1-y)\big)^{g_{k+1}} (1-a(1-y)).
\end{multline*}
This shows that under $Q(\cdot\mid Y=y, K=k)$, the r.v.'s $G_j$, $1\le j\le k+1$, are independent geometric r.v.'s with the same failure probability $a(1-y)$, all shifted but the last one. Note that this result is slightly stronger than the result stated in the lemma.
\end{proof}

\section*{\textsc{Discussion}}

Take a binary branching process satisfying the properties $(\star)$, for example a birth-death process with (possibly time-inhomogeneous) birth and death rates, stopped at time $T$ conditional on $N\ge 1$, where $N$ is the number of alive particles at time $T$. We remember that the reduced tree at height $T$ is a CPP, that is, its node depths form a sequence of iid random variables distributed as $H$, with inverse tail distribution $F$ characterized by \eqref{eqn:characF}, stopped at its first value larger than $T$. In particular, $N$ is a shifted geometric random variable with failure probability $a= P(H<T)$.

When subsampling from $N$ each alive particle independently with the same probability $y$, it is also known that the reduced tree of this so-called Bernoulli sample remains a CPP, with inverse tail distribution $F_y=1-y+yF$. Note that upon conditioning the sample size to equal $k$, the common distribution of node depths of the tree does depend on $y$. This shows in particular that (the genealogy of) a Bernoulli sample conditioned by its size to equal $k$ does not have the same distribution as (the genealogy of) a uniform $k$-sample.

Until recently (see last paragraph below for an account of more recent results), the only known result about the genealogy of a $k$-sample was the likelihood \eqref{eqn:likelihood-k} (which is Eq (18) in \cite{LS13} -- note in passing the misplacement of the binomial coefficient in this equation). In the present paper, we have obtained with \eqref{eqn:distribution-function} a simplified version of Eq (21) in \cite{LS13}, which gives the distribution function of the node depths of the genealogy of a $k$-sample, on the event that the number $N$ of tips in the full tree equals $k+m$. It is possible to sum over $m$ the equations \eqref{eqn:distribution-function} to get an expression for the distribution function of node depths of the $k$-sample involving the hypergeometric function $z\mapsto \sum_{m\ge 0} \frac{m!}{(m+k)!} z^m$, but we have taken another path to characterize this distribution. 

Indeed, in the proof of Theorem \ref{thm:main}, we have shown that we can obtain the genealogy of the $k$-sample as the genealogy of a Bernoulli sample with sampling probability $Y$, where $Y$ is distributed according to the improper prior $y^{-1}dy$ over $(0,1)$, and the Bernoulli sample is further conditioned to be of size $k$. As a result, the node depths of the genealogy of the $k$-sample form a mixture of sequences of $k-1$ iid random variables, where the mixing distribution is the posterior distribution $\mu_k$ of $Y$ given by $\eqref{eqn:muk}$, as seen in Eq \eqref{eqn:likelihood-k}. This representation is reminiscent of de Finetti's theorem, which states that any infinite exchangeable (i.e., invariant in distribution under the action of permutations with finite support) sequence is a mixture of sequences of iid random variables \cite{A85, F31, HS55}. 

Since it is obvious that the $k-1$ node depths of a $k$-sample form an exchangeable sequence, one might think at first sight that Theorem \ref{thm:main} is not so surprising. Nevertheless, let us first underline the fact that here the representation is explicit and has an illuminating interpretation in terms of Bernoulli sampling. Second, this exchangeable sequence is finite, so that there is actually no guarantee \textit{a priori} that such a de Finetti representation exists. This is confirmed by the fact that the mixing distribution $\mu_k$ does depend on $k$, which shows that there is actually no embedding of this finite exchangeable sequence into an infinite exchangeable sequence that would justify our finding \textit{a posteriori}.

We wish to emphasize the implications of Theorem \ref{thm:main} in terms of simulation and statistical inference. As seen in Corollary \ref{cor:cor1}, to simulate the genealogy of a $k$-sample, one can indeed draw first  a r.v. $Y$ from the probability distribution $\mu_k$ and then conditional on $Y=y$, draw $k-1$ iid r.v. $H_1',\ldots, H_{k-1}'$ with common density $c^{-1}f_y$. Then indeed, this $(k-1)$-tuple has the same law as the node depths of the genealogy of the $k$-sample. As for inference purposes, as seen in Corollary \ref{cor:cor2}, the likelihood of a tree under the $k$-sampling scheme has an explicit formula \eqref {eqn:likelihood-k} which can be computed numerically instantaneously. 

Let us finally point out the proximity of our results with those of the recent work \cite{HJR17}, also dealing with the genealogy of a $k$-sample from a branching process. The authors of \cite{HJR17} consider time-homogeneous Markovian branching processes, possibly nonbinary, whereas we consider here possibly non-Markovian and time-inhomogeneous branching processes, but always binary. The spine methods used in \cite{HJR17} could certainly apply to time-inhomogeneous processes, but doubtedly to non-Markovian processes. Similarly, the methods used in the present paper could hardly apply to nonbinary processes (for which, in passing, the only explicit result actually available in \cite{HJR17} is Theorem 2.3, in the case of finite variance and near-critical limit, where all coalescences are actually binary). Our results can thus only agree in the case of birth--death processes. Specifically, replacing $f_y$ in \eqref{eqn:likelihood-k} by its expression \eqref{eqn:characfy} in the special case of a time-homogeneous birth--death process  yields Proposition 5.2 in \cite{HJR17}. 
Finally, let us mention that despite the apparent similarity between the distribution function of node depths \eqref{eqn:distribution-function} on the event $N=k+m$ in the special case of a birth-death process on the one hand, and on the other hand the distribution function in Theorems 2.1 and 2.2 in \cite{HJR17}, we were not able, by summing over $m$, to get rid of the hypergeometric function that appears in our calculations but is absent from the corresponding expressions in   \cite{HJR17}.

\paragraph{Acknowledgments.} The author thanks the {\em Center for Interdisciplinary Research in Biology} (Coll\`ege de France) for funding. The present work is owing to discussions with the authors of \cite{HJR17}.

 \bibliographystyle{abbrv}
\bibliography{myref}

\end{document}

%% file: examplePhylo.tex
\unitlength 1.8mm 
\linethickness{0.2pt}
\ifx\plotpoint\undefined\newsavebox{\plotpoint}\fi 
\begin{picture}(72.13,57)(0,0)
\thicklines
\put(12,27){\line(0,1){12}}
\thinlines
\put(15,36){\line(1,0){.07}}
\put(11.93,35.93){\line(1,0){.75}}
\put(13.43,35.93){\line(1,0){.75}}
\thicklines
\put(15,36){\line(0,1){8}}
\thinlines
\put(18,42){\line(1,0){.07}}
\put(14.93,41.93){\line(1,0){.75}}
\put(16.43,41.93){\line(1,0){.75}}
\thicklines
\put(18,42){\line(0,1){5}}
\thinlines
\put(20,40){\line(1,0){.07}}
\put(14.93,39.93){\line(1,0){.833}}
\put(16.6,39.93){\line(1,0){.833}}
\put(18.26,39.93){\line(1,0){.833}}
\thicklines
\put(20,40){\line(0,1){4}}
\put(23,41){\line(0,1){5}}
\put(29,31){\line(0,1){6}}
\thinlines
\put(32,35){\line(1,0){.07}}
\put(28.93,34.93){\line(1,0){.75}}
\put(30.43,34.93){\line(1,0){.75}}
\thicklines
\put(32,35){\line(0,1){15}}
\thinlines
\put(35,48){\line(1,0){.07}}
\put(31.93,47.93){\line(1,0){.75}}
\put(33.43,47.93){\line(1,0){.75}}
\thicklines
\put(35,48){\line(0,1){3}}
\thinlines
\put(23,41){\line(1,0){.07}}
\put(19.93,40.93){\line(1,0){.75}}
\put(21.43,40.93){\line(1,0){.75}}
\put(38,49){\line(1,0){.07}}
\put(34.93,48.93){\line(1,0){.75}}
\put(36.43,48.93){\line(1,0){.75}}
\thicklines
\put(38,49){\line(0,1){4}}
\thinlines
\put(41,52){\line(1,0){.07}}
\put(37.93,51.93){\line(1,0){.75}}
\put(39.43,51.93){\line(1,0){.75}}
\thicklines
\put(41,52){\line(0,1){3}}
\thinlines
\put(43,51){\line(1,0){.07}}
\put(37.93,50.93){\line(1,0){.833}}
\put(39.6,50.93){\line(1,0){.833}}
\put(41.26,50.93){\line(1,0){.833}}
\thicklines
\put(43,51){\line(0,1){2}}
\thinlines
\put(40,44){\line(1,0){.07}}
\put(31.93,43.93){\line(1,0){.889}}
\put(33.71,43.93){\line(1,0){.889}}
\put(35.49,43.93){\line(1,0){.889}}
\put(37.26,43.93){\line(1,0){.889}}
\put(39.04,43.93){\line(1,0){.889}}
\thicklines
\put(40,44){\line(0,1){4}}
\thinlines
\put(45,38){\line(1,0){.07}}
\put(31.93,37.93){\line(1,0){.929}}
\put(33.79,37.93){\line(1,0){.929}}
\put(35.64,37.93){\line(1,0){.929}}
\put(37.5,37.93){\line(1,0){.929}}
\put(39.36,37.93){\line(1,0){.929}}
\put(41.22,37.93){\line(1,0){.929}}
\put(43.07,37.93){\line(1,0){.929}}
\thicklines
\put(45,38){\line(0,1){8}}
\thinlines
\put(48,44){\line(1,0){.07}}
\put(44.93,43.93){\line(1,0){.75}}
\put(46.43,43.93){\line(1,0){.75}}
\thicklines
\put(48,44){\line(0,1){4}}
\thinlines
\put(50,42){\line(1,0){.07}}
\put(44.93,41.93){\line(1,0){.833}}
\put(46.6,41.93){\line(1,0){.833}}
\put(48.26,41.93){\line(1,0){.833}}
\thicklines
\put(50,42){\line(0,1){4}}
\thinlines
\put(53,45){\line(1,0){.07}}
\put(49.93,44.93){\line(1,0){.75}}
\put(51.43,44.93){\line(1,0){.75}}
\thicklines
\put(53,45){\line(0,1){7}}
\thinlines
\put(58,48){\line(1,0){.07}}
\put(52.93,47.93){\line(1,0){.833}}
\put(54.6,47.93){\line(1,0){.833}}
\put(56.26,47.93){\line(1,0){.833}}
\put(60,46){\line(1,0){.07}}
\put(52.93,45.93){\line(1,0){.875}}
\put(54.68,45.93){\line(1,0){.875}}
\put(56.43,45.93){\line(1,0){.875}}
\put(58.18,45.93){\line(1,0){.875}}
\thicklines
\put(60,46){\line(0,1){3}}
\thinlines
\put(62,39){\line(1,0){.07}}
\put(44.93,38.93){\line(1,0){.944}}
\put(46.82,38.93){\line(1,0){.944}}
\put(48.71,38.93){\line(1,0){.944}}
\put(50.6,38.93){\line(1,0){.944}}
\put(52.49,38.93){\line(1,0){.944}}
\put(54.37,38.93){\line(1,0){.944}}
\put(56.26,38.93){\line(1,0){.944}}
\put(58.15,38.93){\line(1,0){.944}}
\put(60.04,38.93){\line(1,0){.944}}
\thicklines
\put(62,39){\line(0,1){5}}
\thinlines
\put(65,42){\line(1,0){.07}}
\put(61.93,41.93){\line(1,0){.75}}
\put(63.43,41.93){\line(1,0){.75}}
\thicklines
\put(65,42){\line(0,1){3}}
\thinlines
\put(67,33){\line(1,0){.07}}
\put(28.93,32.93){\line(1,0){.974}}
\put(30.88,32.93){\line(1,0){.974}}
\put(32.83,32.93){\line(1,0){.974}}
\put(34.78,32.93){\line(1,0){.974}}
\put(36.72,32.93){\line(1,0){.974}}
\put(38.67,32.93){\line(1,0){.974}}
\put(40.62,32.93){\line(1,0){.974}}
\put(42.57,32.93){\line(1,0){.974}}
\put(44.52,32.93){\line(1,0){.974}}
\put(46.47,32.93){\line(1,0){.974}}
\put(48.42,32.93){\line(1,0){.974}}
\put(50.37,32.93){\line(1,0){.974}}
\put(52.31,32.93){\line(1,0){.974}}
\put(54.26,32.93){\line(1,0){.974}}
\put(56.21,32.93){\line(1,0){.974}}
\put(58.16,32.93){\line(1,0){.974}}
\put(60.11,32.93){\line(1,0){.974}}
\put(62.06,32.93){\line(1,0){.974}}
\put(64.01,32.93){\line(1,0){.974}}
\put(65.96,32.93){\line(1,0){.974}}
\thicklines
\put(67,33){\line(0,1){3}}
\thinlines
\put(29,31){\line(1,0){.07}}
\put(11.93,30.93){\line(1,0){.944}}
\put(13.82,30.93){\line(1,0){.944}}
\put(15.71,30.93){\line(1,0){.944}}
\put(17.6,30.93){\line(1,0){.944}}
\put(19.49,30.93){\line(1,0){.944}}
\put(21.37,30.93){\line(1,0){.944}}
\put(23.26,30.93){\line(1,0){.944}}
\put(25.15,30.93){\line(1,0){.944}}
\put(27.04,30.93){\line(1,0){.944}}
\put(25,37){\line(1,0){.07}}
\put(14.93,36.93){\line(1,0){.909}}
\put(16.75,36.93){\line(1,0){.909}}
\put(18.57,36.93){\line(1,0){.909}}
\put(20.38,36.93){\line(1,0){.909}}
\put(22.2,36.93){\line(1,0){.909}}
\put(24.02,36.93){\line(1,0){.909}}
\thicklines
\put(25,37){\line(0,1){3}}
\thinlines
\put(27,39){\line(1,0){.07}}
\put(24.93,38.93){\line(1,0){.667}}
\put(26.26,38.93){\line(1,0){.667}}
\thicklines
\put(27,39){\line(0,1){2}}
\thinlines
\put(9,27){\vector(0,1){30}}
\put(7.979,26.979){\line(1,0){.9866}}
\put(9.953,26.979){\line(1,0){.9866}}
\put(11.926,26.979){\line(1,0){.9866}}
\put(13.899,26.979){\line(1,0){.9866}}
\put(15.872,26.979){\line(1,0){.9866}}
\put(17.845,26.979){\line(1,0){.9866}}
\put(19.819,26.979){\line(1,0){.9866}}
\put(21.792,26.979){\line(1,0){.9866}}
\put(23.765,26.979){\line(1,0){.9866}}
\put(25.738,26.979){\line(1,0){.9866}}
\put(27.712,26.979){\line(1,0){.9866}}
\put(29.685,26.979){\line(1,0){.9866}}
\put(31.658,26.979){\line(1,0){.9866}}
\put(33.631,26.979){\line(1,0){.9866}}
\put(35.605,26.979){\line(1,0){.9866}}
\put(37.578,26.979){\line(1,0){.9866}}
\put(39.551,26.979){\line(1,0){.9866}}
\put(41.524,26.979){\line(1,0){.9866}}
\put(43.497,26.979){\line(1,0){.9866}}
\put(45.471,26.979){\line(1,0){.9866}}
\put(47.444,26.979){\line(1,0){.9866}}
\put(49.417,26.979){\line(1,0){.9866}}
\put(51.39,26.979){\line(1,0){.9866}}
\put(53.364,26.979){\line(1,0){.9866}}
\put(55.337,26.979){\line(1,0){.9866}}
\put(57.31,26.979){\line(1,0){.9866}}
\put(59.283,26.979){\line(1,0){.9866}}
\put(61.257,26.979){\line(1,0){.9866}}
\put(63.23,26.979){\line(1,0){.9866}}
\put(65.203,26.979){\line(1,0){.9866}}
\put(67.176,26.979){\line(1,0){.9866}}
\put(69.149,26.979){\line(1,0){.9866}}
\put(71.123,26.979){\line(1,0){.9866}}
\put(7.979,1.979){\line(1,0){.9866}}
\put(9.953,1.979){\line(1,0){.9866}}
\put(11.926,1.979){\line(1,0){.9866}}
\put(13.899,1.979){\line(1,0){.9866}}
\put(15.872,1.979){\line(1,0){.9866}}
\put(17.845,1.979){\line(1,0){.9866}}
\put(19.819,1.979){\line(1,0){.9866}}
\put(21.792,1.979){\line(1,0){.9866}}
\put(23.765,1.979){\line(1,0){.9866}}
\put(25.738,1.979){\line(1,0){.9866}}
\put(27.712,1.979){\line(1,0){.9866}}
\put(29.685,1.979){\line(1,0){.9866}}
\put(31.658,1.979){\line(1,0){.9866}}
\put(33.631,1.979){\line(1,0){.9866}}
\put(35.605,1.979){\line(1,0){.9866}}
\put(37.578,1.979){\line(1,0){.9866}}
\put(39.551,1.979){\line(1,0){.9866}}
\put(41.524,1.979){\line(1,0){.9866}}
\put(43.497,1.979){\line(1,0){.9866}}
\put(45.471,1.979){\line(1,0){.9866}}
\put(47.444,1.979){\line(1,0){.9866}}
\put(49.417,1.979){\line(1,0){.9866}}
\put(51.39,1.979){\line(1,0){.9866}}
\put(53.364,1.979){\line(1,0){.9866}}
\put(55.337,1.979){\line(1,0){.9866}}
\put(57.31,1.979){\line(1,0){.9866}}
\put(59.283,1.979){\line(1,0){.9866}}
\put(61.257,1.979){\line(1,0){.9866}}
\put(63.23,1.979){\line(1,0){.9866}}
\put(65.203,1.979){\line(1,0){.9866}}
\put(67.176,1.979){\line(1,0){.9866}}
\put(69.149,1.979){\line(1,0){.9866}}
\put(71.123,1.979){\line(1,0){.9866}}
\thicklines
\put(58,48){\line(0,1){2}}
\thinlines
\put(55,50){\line(1,0){.07}}
\put(52.93,49.93){\line(1,0){.667}}
\put(54.26,49.93){\line(1,0){.667}}
\thicklines
\put(55,50){\line(0,1){3}}
\thinlines
\put(57,51){\line(1,0){.07}}
\put(54.93,50.93){\line(1,0){.667}}
\put(56.26,50.93){\line(1,0){.667}}
\thicklines
\put(57,51){\line(0,1){3}}
\put(6,55.875){\makebox(0,0)[cc]{(a)}}
\put(6,21.75){\makebox(0,0)[cc]{(b)}}
\thinlines
\put(8,43){\line(1,0){62}}
\put(8,18){\line(1,0){62}}
\thicklines
\put(28.5,18){\vector(0,-1){12}}
\put(34.5,18){\vector(0,-1){5}}
\put(30.5,13.125){\makebox(0,0)[cc]{\scriptsize $H_4$}}
\put(36.5,15.625){\makebox(0,0)[cc]{\scriptsize $H_5$}}
\thinlines
\put(21,18){\line(0,-1){3}}
\put(27,18){\line(0,-1){2}}
\put(33,18){\line(0,-1){12}}
\put(45,18){\line(0,-1){1}}
\put(51,18){\line(0,-1){4}}
\put(57,18){\line(0,-1){1}}
\put(39,18){\line(0,-1){5}}
\put(15,18){\line(0,-1){1}}
\put(14.979,16.979){\line(-1,0){.8571}}
\put(13.265,16.979){\line(-1,0){.8571}}
\put(11.551,16.979){\line(-1,0){.8571}}
\put(9.836,16.979){\line(-1,0){.8571}}
\put(20.979,14.979){\line(-1,0){.9231}}
\put(19.133,14.979){\line(-1,0){.9231}}
\put(17.287,14.979){\line(-1,0){.9231}}
\put(15.441,14.979){\line(-1,0){.9231}}
\put(13.595,14.979){\line(-1,0){.9231}}
\put(11.749,14.979){\line(-1,0){.9231}}
\put(9.902,14.979){\line(-1,0){.9231}}
\put(26.979,15.979){\line(-1,0){.8571}}
\put(25.265,15.979){\line(-1,0){.8571}}
\put(23.551,15.979){\line(-1,0){.8571}}
\put(21.836,15.979){\line(-1,0){.8571}}
\put(32.979,5.979){\line(-1,0){.96}}
\put(31.059,5.979){\line(-1,0){.96}}
\put(29.139,5.979){\line(-1,0){.96}}
\put(27.219,5.979){\line(-1,0){.96}}
\put(25.299,5.979){\line(-1,0){.96}}
\put(23.379,5.979){\line(-1,0){.96}}
\put(21.459,5.979){\line(-1,0){.96}}
\put(19.539,5.979){\line(-1,0){.96}}
\put(17.619,5.979){\line(-1,0){.96}}
\put(15.699,5.979){\line(-1,0){.96}}
\put(13.779,5.979){\line(-1,0){.96}}
\put(11.859,5.979){\line(-1,0){.96}}
\put(9.939,5.979){\line(-1,0){.96}}
\put(38.979,12.979){\line(-1,0){.8571}}
\put(37.265,12.979){\line(-1,0){.8571}}
\put(35.551,12.979){\line(-1,0){.8571}}
\put(33.836,12.979){\line(-1,0){.8571}}
\put(44.979,16.979){\line(-1,0){.8571}}
\put(43.265,16.979){\line(-1,0){.8571}}
\put(41.551,16.979){\line(-1,0){.8571}}
\put(39.836,16.979){\line(-1,0){.8571}}
\put(50.979,13.979){\line(-1,0){.9231}}
\put(49.133,13.979){\line(-1,0){.9231}}
\put(47.287,13.979){\line(-1,0){.9231}}
\put(45.441,13.979){\line(-1,0){.9231}}
\put(43.595,13.979){\line(-1,0){.9231}}
\put(41.749,13.979){\line(-1,0){.9231}}
\put(39.902,13.979){\line(-1,0){.9231}}
\put(56.979,16.979){\line(-1,0){.8571}}
\put(55.265,16.979){\line(-1,0){.8571}}
\put(53.551,16.979){\line(-1,0){.8571}}
\put(51.836,16.979){\line(-1,0){.8571}}
\put(15,20){\makebox(0,0)[cc]{\scriptsize $1$}}
\put(21,20){\makebox(0,0)[cc]{\scriptsize $2$}}
\put(27,20){\makebox(0,0)[cc]{\scriptsize $3$}}
\put(33,20){\makebox(0,0)[cc]{\scriptsize $4$}}
\put(39,20){\makebox(0,0)[cc]{\scriptsize $5$}}
\put(45,20){\makebox(0,0)[cc]{\scriptsize $6$}}
\put(51,20){\makebox(0,0)[cc]{\scriptsize $7$}}
\put(57,20){\makebox(0,0)[cc]{\scriptsize $8$}}
\put(14.5,44.5){\makebox(0,0)[cc]{\scriptsize $0$}}
\put(19,45.75){\makebox(0,0)[cc]{\scriptsize $1$}}
\put(20.875,44.25){\makebox(0,0)[cc]{\scriptsize $2$}}
\put(24,44.125){\makebox(0,0)[cc]{\scriptsize $3$}}
\put(30.875,44.375){\makebox(0,0)[cc]{\scriptsize $4$}}
\put(43.875,44){\makebox(0,0)[cc]{\scriptsize $5$}}
\put(51.375,43.75){\makebox(0,0)[cc]{\scriptsize $6$}}
\put(61.125,44.125){\makebox(0,0)[cc]{\scriptsize $7$}}
\put(66,44){\makebox(0,0)[cc]{\scriptsize $8$}}
\put(9,2){\line(0,1){16}}
\put(6.375,17.875){\makebox(0,0)[cc]{\scriptsize $T$}}
\put(9,20){\makebox(0,0)[cc]{\scriptsize $0$}}
\put(6.5,43){\makebox(0,0)[cc]{\scriptsize $T$}}
\put(6.25,2){\makebox(0,0)[cc]{\scriptsize $0$}}
\put(6.25,27){\makebox(0,0)[cc]{\scriptsize $0$}}
\end{picture}